\documentclass{birkjour}

\usepackage{amsmath,amsthm,amssymb,amsfonts,xcolor}

 \newtheorem{thm}{Theorem}
 
 \newtheorem{lem}[thm]{Lemma}
 
 \theoremstyle{definition}
 
 \theoremstyle{remark}
 \newtheorem{rem}[thm]{Remark}

 \numberwithin{equation}{section}

\DeclareSymbolFont{bbold}{U}{bbold}{m}{n}
\DeclareSymbolFontAlphabet{\mathbbold}{bbold}
\newcommand{\1}{\mathbbold{1}}

\newcommand{\C}{\mathbb C}
\newcommand{\R}{\mathbb R}

\newcommand\cl{\operatorname{cl}}
\newcommand\Id{\mathrm{I}}
\newcommand\la{\lambda}
\newcommand{\rfrac}[2]{\tfrac{#1}{\raisebox{0.1em}{\scriptsize$#2$}}}

\newcommand{\sk}[2]{\langle #1,#2\rangle}
\newcommand{\bsk}[2]{\bigl\langle #1,#2\bigr\rangle}
\newcommand{\nn}[2]{\norm{#1}_{#2}}

\makeatletter
\newcommand*\if@single[3]{%
  \setbox0\hbox{${\mathaccent"0362{#1}}^H$}%
  \setbox2\hbox{${\mathaccent"0362{\kern0pt#1}}^H$}%
  \ifdim\ht0=\ht2 #3\else #2\fi
  }
\newcommand*\rel@kern[1]{\kern#1\dimexpr\macc@kerna}
\newcommand*\widebar[1]{\@ifnextchar^{{\wide@bar{#1}{0}}}{\wide@bar{#1}{1}}}
\newcommand*\wide@bar[2]{\if@single{#1}{\wide@bar@{#1}{#2}{1}}{\wide@bar@{#1}{#2}{2}}}
\newcommand*\wide@bar@[3]{%
  \begingroup
  \def\mathaccent##1##2{%
    \if#32 \let\macc@nucleus\first@char \fi
    \setbox\z@\hbox{$\macc@style{\macc@nucleus}_{}$}%
    \setbox\tw@\hbox{$\macc@style{\macc@nucleus}{}_{}$}%
    \dimen@\wd\tw@
    \advance\dimen@-\wd\z@
    \divide\dimen@ 3
    \@tempdima\wd\tw@
    \advance\@tempdima-\scriptspace
    \divide\@tempdima 10
    \advance\dimen@-\@tempdima
    \ifdim\dimen@>\z@ \dimen@0pt\fi
    \rel@kern{0.6}\kern-\dimen@
    \if#31
      \overline{\rel@kern{-0.6}\kern\dimen@\macc@nucleus
                \rel@kern{0.4}\kern\dimen@}%
      \advance\dimen@0.4\dimexpr\macc@kerna
      \let\final@kern#2%
      \ifdim\dimen@<\z@ \let\final@kern1\fi
      \if\final@kern1 \kern-\dimen@\fi
    \else
      \overline{\rel@kern{-0.6}\kern\dimen@#1}%
    \fi
  }%
  \macc@depth\@ne
  \let\math@bgroup\@empty \let\math@egroup\macc@set@skewchar
  \mathsurround\z@ \frozen@everymath{\mathgroup\macc@group\relax}%
  \macc@set@skewchar\relax
  \let\mathaccentV\macc@nested@a
  \if#31
    \macc@nested@a\relax111{#1}%
  \else
    \def\gobble@till@marker##1\endmarker{}%
    \futurelet\first@char\gobble@till@marker#1\endmarker
    \ifcat\noexpand\first@char A\else
      \def\first@char{}%
    \fi
    \macc@nested@a\relax111{\first@char}%
  \fi
  \endgroup
}

\def\@row#1,{#1\@ifnextchar;{\@gobble}{&\@row}}
\def\@matrix{%
    \expandafter\@row\my@arg,;%
    \@ifnextchar({\\ \get@in@paren{\@matrix}}{\after@matrix}%
    }
\def\matrixtype#1#2#3{%
    \ifmmode\def\after@matrix{\end{#2}\right#3}%
    \else\def\after@matrix{\end{#2}\right#3}$\fi
    \left#1\begin{#2}\get@in@paren{\@matrix}%
    }
\def\@column#1,{#1\@ifnextchar;{\@gobble}{\\ \@column}}
\newcommand\vect{}
\def\svect(#1){\hbox{$\m@th\nulldelimiterspace=0pt
\left(\!\begin{smallmatrix}\@column#1,;\end{smallmatrix}\!\right)$}}
\def\vect{\get@in@paren{\@vect}}
\def\@vect{\left(\begin{matrix}\expandafter\@column\my@arg,;\end{matrix}\right)}
\def\get@in@paren#1({\def\my@arg{}\def\my@rest{}\def\after@get{#1}\get@arg}
\let\e@a\expandafter
\def\get@arg#1){\e@a\kl@test\my@rest#1(;}
\def\kl@test#1(#2;{\e@a\def\e@a\my@arg\e@a{\my@arg#1}%
                   \ifx:#2:\let\my@exec\after@get
                   \else\let\my@exec\get@arg
                        \e@a\def\e@a\my@arg\e@a{\my@arg(}%
                        \def@rest#2;%
                   \fi\my@exec}
\def\def@rest#1(;{\def\my@rest{#1\kl@zu}}
\def\kl@zu{)}
\makeatother

\newcommand\smat{\matrixtype({smallmatrix})}

\renewcommand\le{\leqslant}
\renewcommand\ge{\geqslant}

\newcommand{\breakOK}{\penalty20}
\newcommand{\set}[2]{\bigl\{#1\,\,{:}\,\breakOK\;#2\bigr\}}

\newcommand{\Ell}[1]{\mathrm{L}_{#1}}
\newcommand{\emdf}{\bfseries}
\newcommand{\dom}{\mathrm{dom}}
\newcommand{\abs}[1]{|#1|}
\newcommand{\norm}[1]{\|#1\|\rule[-0.4ex]{0pt}{0pt}}
\newcommand{\ud}{\mathrm{d}}
\newcommand{\konj}[1]{\widebar{#1}}
\newcommand{\ue}{\mathrm{e}}

\newcommand{\vphi}{\varphi}

\begin{document}

\title[Numerical range of generators]{On the numerical range of generators 
of\\ symmetric $\Ell{\infty}$-contractive semigroups}

\author[M.~Haase]{Markus Haase}

\address{%
Mathematisches Seminar\\
Christian-Albrechts-Universit\"at zu Kiel\\
24089 Kiel\\
Germany}

\email{haase@math.uni-kiel.de}

\author[P.C.~Kunstmann]{Peer Christian Kunstmann}
\address{Institut f\"ur Analysis\\ Karlsruher Institut f\"ur
Technologie (KIT)\\ 76128 Karlsruhe\\ Germany}

\email{peer.kunstmann@kit.edu}

\author[H.~Vogt]{Hendrik Vogt}
\address{
Fachbereich 3 - Mathematik\\
Universit\"at Bremen\\
28359 Bremen\\
Germany}

\email{hendrik.vo\rlap{\textcolor{white}{hugo@egon}}gt@uni-\rlap{\textcolor{white}{%
hannover}}bremen.de}



\date{\today}

\begin{abstract}
A result by Liskevich and Perelmuter from 1995 yields 
the optimal angle of analyticity for
symmetric submarkovian semigroups on $\Ell{p}$, $1<p<\infty$. 
C.~Kriegler showed in 2011 that the result remains true 
without the assumption of positivity of the semigroup.  
Here we give an elementary proof of Kriegler's result.
\end{abstract}

\maketitle

\section{Introduction}\label{s.intro}

Let $(\Omega,\mu)$ be a measure space and let $A$ be a positive self-adjoint  
operator in $\Ell{2}(\Omega,\mu)$. Then $-A$ generates a strongly continuous 
semigroup $(T(t))_{t\ge0}$ which extends analytically  to a contraction 
semigroup on the open right half plane.  Such a  semigroup is called a
{\emdf symmetric $\Ell{\infty}$-contractive semigroup}%
\footnote{Such semigroups are called {\em diffusion semigroups}
in \cite{kriegler} 
and {\em symmetric contraction semigroups} in \cite{carb-drag}.}
if, in addition, one has
\[
 \nn{T(t)f}{\infty}\le\nn{f}{\infty},\quad \text{for all}\,\,
 f\in \Ell{2}(\Omega,\mu)
\cap \Ell{\infty}(\Omega,\mu).
\]
Then, by symmetry, the semigroup is also $\Ell{1}$-contractive, 
and by interpolation one obtains for each $1<p<\infty$ a consistent 
$C_0$-semigroup $(T_p(t))_{t\ge0}$ of contractions on $\Ell{p}(\Omega,\mu)$. 
Let the generator of this semigroup be denoted by $-A_p$, with domain
$\dom(A_p)$.  

In order to state the main result we define for $p \in [1,\infty)$ 
the mapping
\[ F_p : \C \to \C ,\qquad 
F_p(z) := \begin{cases}  z\abs{z}^{p-2} & \text{if}\,\, z\neq 0,\\
 0 & \text{if}\,\, z= 0.
\end{cases}
\]
Note that for $f\in \Ell{p}(\Omega,\mu)$ with $\norm{f}_p = 1$
the function $F_p(f) := F_p\circ f$
has the properties
\[  \norm{F_p(f)}_{p'} = 1,\qquad \int_\Omega f \cdot \konj{F_p(f)}\, \ud{\mu} = 1,
\]
where $p'$ is the dual exponent, i.e., $\frac{1}{p} + \frac{1}{p'} =1$. 
In case that $p >1$, $F_p(f)$ is uniquely characterized by these properties, 
and the {\emdf numerical range} of $A_p$ is the set of numbers
\[  \int_\Omega  A_pf \cdot \konj{F_p(f)}\,
\ud{\mu}, \qquad  \text{where}\quad f\in \dom(A_p),\ \norm{f}_p = 1.
\]
For $0 \le \vphi \le \frac{\pi}{2}$ we define the sector
\[
  \Sigma(\vphi) := \set{z\in\C\setminus\{0\}}{\mathopen|\arg z| \le \vphi} \cup \{0\}
\]
and call $\vphi$ its {\emdf opening angle}.
For $p\in(1,\infty)$, let $\Sigma_p := \Sigma(\vphi_p)$, where
\[ \vphi_p := \arcsin|1-\rfrac2p|.
\]
Note that $\vphi_{p'} = \vphi_p$. Moreover,
if $p$ passes from $2$ to $1$ or to $\infty$, then
$\abs{1- \rfrac{2}{p}}$ passes from $0$ to $1$ and the opening angle 
$\vphi_p =  \arcsin\abs{1- \rfrac{2}{p}}$ of $\Sigma_p$
passes from $0$ to $\frac{\pi}{2}$.
We point out that $\vphi_p$ is smaller than the angle $\frac\pi2|1-\rfrac2p|$
that arises from interpolation between $0$ and $\frac{\pi}{2}$.
A short computation reveals that
$\vphi_p$ has the (often used) alternative representation 
\[ \vphi_p = \arcsin\abs{1- \rfrac{2}{p}}
= \arctan \tfrac{\abs{p-2}}{2 \sqrt{p-1}}.
\]
Now, here is the main result.

\begin{thm}\label{thm1} 
Let $p \in (1, \infty)$ and let $-A_p$ be the generator 
on $\Ell{p}$ of a symmetric $\Ell{\infty}$-contractive semigroup on 
$\Ell{2}(\Omega,\mu)$. 
Then the numerical range of $A_p$ is contained in the sector 
$\Sigma_p$, and $(T_p(t))_{t\ge0}$ 
extends to an analytic contraction semigroup on the sector with 
opening angle $\arccos|1-\rfrac{2}{p}|$. 
\end{thm}

Under the additional assumption that the semigroup 
consists of pos\-i\-tiv\-ity-pre\-serv\-ing
operators, Theorem~\ref{thm1}
is due to Liskevich and Perelmuter \cite{lisk-perel}.
The full result was established by Kriegler in \cite{kriegler} 
in the framework of  non-commutative operator theory. 
Recently, Theorem~\ref{thm1} has been recovered 
by Carbonaro and Dragi\v{c}evi\'c in \cite{carb-drag}
as a corollary of 
much stronger results. 
In \cite{haase}, the first-named author streamlined and extended 
some of the methods used in \cite{carb-drag} and showed that 
Theorem~\ref{thm1} can be deduced easily without making use of
Bellman functions (which feature prominently in Carbonaro and 
Dragi\v{c}evi\'c's work). 

\medskip

In the following, we shall present an essentially elementary
proof of Theorem~\ref{thm1} extending the arguments from  
\cite{lisk-perel}.
The relation to the other proofs
shall be explained in Section~\ref{s.comparison} below. We note that 
the second assertion in Theorem~\ref{thm1} follows from the first
by virtue of the Lumer-Phillips theorem and the exponential formula
$T_p(t) = \operatorname{s-lim}_{n\to\infty} (\Id + \rfrac{t}{n} A_p)^{-n}$,
cf.\ \cite[Theorem~3.14]{Isem1415}. Hence, it suffices to prove the first assertion.

\section{A two-dimensional special case}

Consider the special case $\Omega = \{1, 2\}$
with measure $\mu = \delta_1 + \delta_2$ and the matrix
\[  A =  \begin{pmatrix}  1 & -1\\ -1 & 1 \end{pmatrix}.
\]
Then $\Ell{p}(\Omega,\mu) = \C^2$ with the usual $p$-norm and   
\[ \ue^{-tA} = \tfrac{1}{2}
\begin{pmatrix}  1+ \ue^{-2t} & 1 - \ue^{-2t} \\ 1 - \ue^{-2t}  & 
1 + \ue^{-2t} \end{pmatrix}\quad \text{for $t\ge 0$}.
\]
Hence, $-A$ generates a (positivity-preserving) 
symmetric $\Ell{\infty}$-contractive semigroup.
For this special case, Theorem~\ref{thm1} reduces to the 
assertion
\begin{equation}\label{scalar-inclusion}
(w-z) \cdot \konj{F_p(w) - F_p(z)} \in  \Sigma_p \quad \text{for all 
$z,w\in \C$},
\end{equation}
which will be established with the next lemma. Moreover, Lemma \ref{LP-lem}
also shows that 
the sector $\Sigma_p$ in Theorem~\ref{thm1} is  optimal
already in this special case.

\begin{lem}\label{LP-lem}
For all $p\in(1,\infty)$ and $z,w\in\C$ one has
\begin{equation}\label{scalar-identity}
  \cl\set{ (w-z) \cdot \konj{F_p(w) - F_p(z)} }{ z,w \in \C } = \Sigma_p.
\end{equation}
\end{lem}

The inclusion $\subseteq$  in Lemma~\ref{LP-lem} has been 
proved originally by Liskevich and Perelmuter
\cite[Lemma~2.2]{lisk-perel}.
We include a new proof that 
helps to understand the appearance of the angle $\vphi_p$.

\begin{proof}
Fix $p \in (1, \infty)$ and write $F = F_p$. To establish~\eqref{scalar-inclusion}
we can restrict to the case that
$0$ is not on the line segment joining $z$ and $w$;
otherwise, $(w-z) \cdot \konj{F_p(w) - F_p(z)} \geq 0.$
We identify, as usual, 
$\C$ with $\R^2$, and note that
$F$ is continuously $\R$-differentiable on $\R^2 \setminus \{0\}$.
Hence, abbreviating $h = w-z$, we obtain
\[ F(z+h) - F(z) = \int_0^1 F'(z + th)h \, \ud{t},
\]
where $F'$ is the Jacobian matrix  of $F$.

Since $\Sigma_p$ is a closed convex cone (note that $\vphi_p \le \frac{\pi}{2}$),
it suffices to prove that
\[ h \cdot \konj{F'(y)h} \in \Sigma_p \quad
\text{for all}\ h\in\R^2,\ y \in \R^2 \setminus \{0\}.
\]
Now, a short elementary computation yields, for $0 \neq  y \in \R^2$,
\[ F'(y) = \abs{y}^{p-2} A_y,
\]
where $A_y := \mathrm{I}_2 + \frac{p{-}2}{\abs{y}^2}\, yy^t$. 
The matrix $A_y$ is symmetric and  has eigenvalues $1$ and $p-1>0$.
(Indeed, $A_yy = (p-1)y$ and $A_yz=z$ for all $z\perp y$.) 
Thus, by Lemma~\ref{geom-lem} below, 
\[ h \cdot \konj{F'(y)h}
\in \Sigma\bigl(\arcsin\rfrac{|p-2|}{p}\bigr) = \Sigma_p
\] 
for all $h\in\R^2$, and this concludes the proof of \eqref{scalar-inclusion}, i.e.,
the inclusion ``$\subseteq$'' in \eqref{scalar-identity}.

For the converse inclusion we denote 
\[ \Sigma := \cl\set{ (w-z) \cdot \konj{F(w) - F(z)} }{ z,w \in \C}.
\]
Since $F(tz) = t^{p{-}1}F(z)$ for all $t > 0$ and $z\in \C$ the set $\Sigma$ is 
invariant
under multiplication with $t > 0$, i.e., a cone. Hence, 
for all $z,h \in \C\setminus\{0\}$ and $t>0$ we obtain
\[
  \frac1t h \cdot \konj{F(z+th)-F(z)}
  = \frac{1}{t^2} (th) \cdot \konj{F(z+th)-F(z)} \in  \Sigma.
\]
Letting $t \searrow 0$ we arrive at
$h \cdot \konj{F'(z)h} \in \Sigma$,
and another application of Lemma~\ref{geom-lem} completes the proof.
\end{proof}

\begin{lem}\label{geom-lem}
Let $A\in\R^{2\times2}$ be a symmetric matrix with 
eigenvalues $1$ and $\lambda>0$. Then
\[
  \set{ h\cdot\konj{Ah} }{ h\in\C }
  = \Sigma\bigl( \arcsin\tfrac{|\lambda-1|}{\lambda+1} \bigr).
\]
\end{lem}

\begin{proof}
Note that $\set{ h\cdot\konj{Ah} }{ h\in\C }$ is a cone in $\C$.
Thus it suffices to show that
\begin{equation}\label{args}
I := \set{ \arg\bigl( h\cdot\konj{Ah}\mkern1mu \bigr) }{ h\in\C\setminus\{0\} }
= \bigl[-\arcsin\tfrac{|\lambda -1|}{\lambda +1},\arcsin\tfrac{|\lambda -1|}{\lambda +1}\bigr]. 
\end{equation}
Now observe that for $h \neq 0$,\,
$\arg\bigl( h\cdot\konj{Ah}\mkern1mu \bigr)$
equals $\sphericalangle(Ah,h)$, the signed angle between $Ah$ and $h$. 
As a consequence, one may suppose without loss of generality
that $A=\smat(1, 0)(0, \lambda)$.
Then $I = \set{ \mkern-1mu\sphericalangle\bigl( A\svect(1,x), \svect(1,x) \bigr) }{ x\in\R }$ 
since $\sphericalangle\bigl( A\svect(0,1), \svect(0,1) \bigr) = 0$.
Setting $a := \arctan x$  and  $b := \arctan(\lambda x)$ we obtain
\[
  \alpha_x := \sphericalangle\bigl( A\svect(1,x), \svect(1,x) \bigr)  
= \sphericalangle\bigl( \svect(1,\lambda x), \svect(1,x) \bigr)
= a-b
\]
and, by virtue of the addition formula for the sine, 
\[
  (1 \pm \lambda) x 
  = \tan a \pm \tan b
  = \frac{\sin a}{\cos a} \pm \frac{\sin b}{\cos b}
  = \frac{\sin(a \pm b)}{\cos a \cos b}\,.
\]
Hence,
\[
  \sin\alpha_x = \sin(a-b) = \sin(a+b) \cdot \frac{1-\lambda}{1+\lambda}\,.
\]
Note that the angle $a+b$ passes from $-\pi$ to $\pi$
as $x$ passes from $-\infty$ to $\infty$.
Thus we obtain the identity~\eqref{args}, and the proof is complete.
\end{proof}

\begin{rem}
A more geometric way to prove Lemma~\ref{geom-lem} consists in 
applying the law of the sines in the triangles
$\triangle OBC$ and $\triangle OB'C$, where the points $A, B, B', C$ and $O$
are defined as
$A := (1,0)$, $B := (1,x)$, $B':= (1, -x)$, 
$C := (1, \lambda x)$ and 
$O := (0,0)$.  (The angle of interest $\alpha_x$ appears at $O$ in the 
triangle $\triangle BOC$.) 
\end{rem}

\section{Proof of Theorem 1}\label{proof}

Let us now turn to the proof of Theorem~\ref{thm1}.

\begin{proof}[Proof of Theorem~\ref{thm1}]
Fix $p\in(1,\infty)$ and write $\sk{f}{g}=\int_\Omega f\konj{g}\,\ud{\mu}$
for $f\in \Ell{p}$ and $g\in \Ell{p'}$, $\frac{1}{p} + \frac{1}{p'} = 1$.  
As above, we abbreviate $F(z) = F_p(z)$. 

As noted already, the second assertion of Theorem~\ref{thm1} follows
from the first by virtue of the Lumer--Phillips theorem.
Hence, we have to show that
\[
  \sk{A_pf}{F(f)}\in\Sigma_p,\quad \text{for all}\quad f\in \dom(A_p).
\]
For this it suffices\footnote{It is also necessary since 
$-(\Id - T(t))$ is again the generator of a symmetric $\Ell{\infty}$-contractive 
semigroup
on $\Ell{2}(\Omega,\mu)$, see \cite[Section 3.1]{haase}.}
to show
\[
\bsk{(\Id-T_p(t))f}{F(f)}\in\Sigma_p,\quad \text{for all}\quad 
f\in \Ell{p}(\Omega,\mu),\ t>0,
\]
since one can divide by $t$ and let $t\searrow 0$. 
Moreover, it is sufficient to check this for the dense subset
$\mathcal{D}$ of step functions
\begin{equation}\label{step-function}
f =\sum_{j=1}^n c_j\1_{B_j}
\end{equation}
where the sets $B_j$ are pairwise disjoint measurable sets 
of positive and finite measure and $c_j\in\C\setminus\{0\}$. 
(In order to see this, take an arbitrary $f\in \Ell{p}$ 
and a sequence $(f_n)_n$ of step functions with $\norm{f_n - f}_p \to 0$,
$f_n \to f$ almost everywhere and absolutely dominated 
by some $0 \le g\in \Ell{p}$. Then $F(f_n) \to F(f)$ almost everywhere
and absolutely dominated by $g^{p-1}$, hence in 
$\Ell{p'}$-norm.)\footnote{Combining this with an argument involving  
subsequences shows that the mapping $f \mapsto F(f)$ is  continuous
from $\Ell{p}$ to $\Ell{p'}$.}

Fix $t>0$ and $f$ as in~\eqref{step-function}, so that 
$F(f) =\sum_k F(c_k)\1_{B_k}$.  Define 
$d_j:=\sk{\1_{B_j}}{\1_{B_j}}=\mu(B_j)$ and $a_{kj} = 
\sk{T(t)\1_{B_j}}{\1_{B_k}}$ for  $1\le j, k \le n$. Then
\begin{align*}
 & \bsk{(I-T(t))f}{F(f)} = {\sum}_{jk} c_j\konj{F(c_k)}
\bsk{(I-T(t))\1_{B_j}}{\1_{B_k}}
\\ & \quad = 
{\sum}_j d_j c_j \konj{F(c_j)} -  {\sum}_{jk} c_j\konj{F(c_k)} a_{kj}
\\ & \quad =  {\sum}_j  \Bigl(d_j - {\sum}_k \abs{a_{kj}}\Bigr)\, c_j \konj{F(c_j)}
\,+\,  {\sum}_{jk} \Bigl(c_j \konj{F(c_j)} \abs{a_{kj}} - c_j \konj{F(c_k)} a_{kj}\Bigr).
\end{align*}
  We claim that the first sum satisfies
\[  {\sum}_j  \Bigl(d_j - {\sum}_k \abs{a_{kj}}\Bigr)\, c_j \konj{F(c_j)} \ge 0.
\]
Since $c_j \konj{F(c_j)} = \abs{c_j}^p \ge 0$, it suffices to
show that $\sum_{k} \abs{a_{kj}} \le d_j$. Choose $\la_{kj}$ such that  
$|\la_{kj}|=1$ and $a_{kj}=\la_{kj}|a_{kj}|$.
Then
\[
  {\sum}_k |a_{kj}| ={\sum}_k\konj{\la_{kj}}\bsk{T(t)\1_{B_j}}{\1_{B_k}}
=\bsk{T(t)\1_{B_j}}{{\textstyle\sum}_k\la_{kj} \1_{B_k}}, 
\]
and hence $\sum_k |a_{kj}| \le \nn{T(t)\1_{B_j}}{1} 
\norm{\sum_k\la_{kj} \1_{B_k}}_\infty \le \norm{\1_{B_j}}_1 = d_j$, since $T(t)$ is an 
$\Ell{1}$-contraction.

In order to deal with the second sum, we note that, by symmetry,
\[
a_{jk} =\sk{T(t)\1_{B_k}}{\1_{B_j}}=\sk{\1_{B_k}}{T(t)\1_{B_j}}= \konj{a_{kj}}.
\]
Therefore and since $\lambda_{kj}F(c_j) = F(\lambda_{kj}c_j)$, 
\begin{align*}
\sum_{j,k} & \Bigl( c_j \konj{F(c_j)} \abs{a_{kj}} - 
c_j\konj{F(c_k)}a_{kj} \Bigr)\\
&=  \frac{1}{2} 
\sum_{j,k} \Bigl(c_j \konj{F(c_j)} \abs{a_{kj}} - 
c_j\konj{F(c_k)}a_{kj} + 
c_k \konj{F(c_k)} \abs{a_{kj}} - 
c_k\konj{F(c_j)} \konj{a_{kj}} \Bigr) 
\\ & = 
\frac{1}{2} 
\sum_{j,k} \abs{a_{kj}} 
(\lambda_{kj}c_j  - c_k) \bigl(\konj{F(\lambda_{kj}c_j)} - \konj{F(c_k)}\mkern1mu\bigr)
\in \Sigma_p
\end{align*}
by Lemma~\ref{LP-lem}.
This concludes the proof.
\end{proof}

\section{Relation to the Existing Proofs}\label{s.comparison}

Our elementary proof proceeds basically along the same reduction lines
as the proof in \cite{haase}. In fact, the main ingredient in the  proof given above
was the fact (established in Lemma~\ref{LP-lem})  that 
\[ (\lambda w -z) \cdot \konj{\lambda F(w) -F(z)} \in \Sigma_p
\]
for all $w,z,\lambda \in \C$ with $\abs{\lambda}=1$.  
A short computation reveals
that this is actually equivalent to Theorem 1 being valid for
the special cases
\[  A =  \begin{pmatrix}  1 & -\konj{\lambda} \\ -\lambda & 1 \end{pmatrix}.
\]
where $\lambda \in \C$ with $\abs{\lambda} =1$. 

The main difference to the paper \cite{haase} is that here
we perform an immediate reduction to a finite atomic measure space
similar as in \cite{lisk-perel}, where
\cite{haase}, following \cite{carb-drag}, takes the detour
via a compact model.
To make this precise, let us consider 
as above the function $f$ as in~\eqref{step-function}
and define the atomic measure space $\Omega' := \{ 1, \dots, n\}$
with $\mu' = \sum_{j=1}^n \mu(B_j) \delta_ {\{j\}}$. On $\Ell{2}(\Omega',\mu')$ consider
the matrix
\[ S = (\tfrac{a_{jk}}{\mu(B_j)})_{j,k}
\]
where $a_{jk} = \sk{T(t)\1_{B_k}}{\1_{B_j}}$ as above. Let
$v := (c_1\,\, \dots\,\, c_n)^t$; then a short computation reveals that
\begin{equation}\label{reduction}
 \bsk{(\Id - S)v}{F(v)}_{\Ell{2}(\Omega',\mu')} 
= \bsk{(\Id - T(t))f}{F(f)}_{\Ell{2}(\Omega,\mu)}.
\end{equation}
The operator $S$ can be written as $J^*T(t)J$, where 
\[ J: \Ell{2}(\Omega',\mu') \to \Ell{2}(\Omega,\mu) ,\quad 
(c_j)_j \mapsto \sum_{j=1}^n c_j \1_{B_j}
\]
is the natural isometric lattice embedding and $J$ is its Hilbert space 
adjoint. (Note that from this observation it is straightforward that $S$ is
an $\Ell{1}$-contraction, a fact that has been proved in Section~\ref{proof} by direct
computation.)

Identity~\eqref{reduction}
implies that Theorem~\ref{thm1} is true in general
if it is true for finite atomic measure spaces. Such spaces are in particular
compact, and the remaining part of the proof in the previous section
is nothing but an adaptation of the proof of  \cite[Theorem 4.15]{haase}
to this special situation.

\begin{rem}
It is straightfoward to conjecture that also the general results
of \cite{haase}, Theorem 2.2-2.4, can be proved
by a direct reduction to finite atomic measure spaces and
avoiding the use of compact models and the sophisticated
operator theory presented in Section 4 of \cite{haase}.
This is indeed true, and will be the topic of a future publication.
\end{rem}

\end{document}